\documentclass[letterpaper, 11pt]{amsart}
\usepackage[letterpaper, left=3cm,right=3cm, top=3cm, bottom=3cm]{geometry}
\usepackage{amsthm}
\usepackage{amsmath}
\usepackage{amssymb}
\usepackage{graphicx}
\usepackage{multicol}
\usepackage{bbm}
\usepackage[
colorlinks=true, citecolor=blue, linkcolor=blue, urlcolor=red]{hyperref}
\usepackage{color}
\usepackage[all]{xy}
\allowdisplaybreaks
\usepackage{mathrsfs}
\usepackage{tikz}
\usetikzlibrary{cd}
\usepackage[all]{xy}
\usepackage{here}

\makeatletter

\@addtoreset{equation}{section}
\makeatother

\makeatletter
\@namedef{subjclassname@2020}{%
  \textup{2020} Mathematics Subject Classification}
\makeatother

\newtheoremstyle{my theoremstyle}
{1.0em}                    
    {1.0em}                    
    {\itshape}                   
    {}                           
    {\scshape}                   
    {.}                          
    {.5em}                       
    {}  

\newtheoremstyle{dfn}
{1.0em}                    
    {1.0em}                    
    {}                   
    {}                           
    {\scshape}                   
    {.}                          
    {.5em}                       
    {}  
    
\theoremstyle{my theoremstyle}
   \newtheorem{thm}{Theorem}[section]
   \newtheorem{lem}[thm]{Lemma}
   \newtheorem{prop}[thm]{Proposition}

\theoremstyle{dfn}
   \newtheorem{dfn}[thm]{Definition}
   
\theoremstyle{remark}

\renewcommand{\labelenumi}{(\roman{enumi})}

\newcommand{\dsum}{\displaystyle \sum}

\newcommand{\C}{\mathbb{C}}
\newcommand{\Q}{\mathbb{Q}}
\newcommand{\Z}{\mathbb{Z}}

\newcommand{\R}{\mathbb{R}}

\renewcommand{\P}{\mathbb{P}}
\renewcommand{\a}{\alpha}

\renewcommand{\l}{\lambda}

\newcommand{\spec}{\operatorname{Spec}}

\newcommand{\rk}{\operatorname{rank}}

\newcommand{\ch}{\operatorname{CH}}

\newcommand{\Aut}{{\mathrm{Aut}}}

\numberwithin{equation}{section}

\date{\today}

\begin{document}
\title{Construction of higher Chow cycles on cyclic coverings of $\P^1 \times \P^1$, Part II}
\date{\today}
\author{Yusuke Nemoto}
\address{Graduate School of Science, Chiba University, 
Yayoicho 1-33, Inage, Chiba, 263-8522 Japan.}
\email{y-nemoto@waseda.jp}
\author{Ken Sato}
\address{Department of Mathematics, Institute of Science Tokyo, 2-12-1 Ookayama, Meguro-ku, Tokyo 152-8551, Japan.}
\email{sato.k.da@m.titech.ac.jp}


\maketitle

\begin{abstract}
In this paper, we construct higher Chow cycles of type $(2, 1)$ on a family of surfaces related to a product of curves, which are certain degree $N$ abelian covers of $\mathbb{P}^1$  branched over $n+2$ points. 
We prove that for a very general member, these cycles generate a subgroup of the indecomposable part of $\operatorname{rank} \ge n\cdot \varphi(N)$, where $\varphi(N)$ is Euler's totient function, by computing their images under the transcendental regulator map.

\end{abstract}

\section{Introduction}
In our previous work \cite{NS}, we constructed higher Chow cycles of type $(2,1)$ on the minimal model of a quotient surface of a product of hypergeometric curves
\begin{equation*}
y^N = x^A(1-x)^A(1-\l_1 x)^A\quad \text{and } \quad y^N = x^{N-A}(1-x)^{N-A}(1-\l_2 x)^{N-A}.
\end{equation*}
Under some conditions, we described the images of the cycles under the Beilinson regulator map in terms of special functions satisfying inhomogeneous Gauss hypergeometric differential equations.
Using these descriptions, we proved that these cycles generate a subgroup of rank $36\cdot \varphi(N)$ of the indecomposable part of the higher Chow group for very general $\l_1,\l_2$, where $\varphi(N)$ is Euler's totient function (loc. cit. Theorem 7.5).

The aim of this paper is to prove similar results for surfaces constructed from hypergeometric curves of the more general form
\begin{equation}\label{genHGC}
y^N=(x-\l)^{a_0}(x-c_1)^{a_1} \cdots (x-c_n)^{a_n} \quad (\l \in \P^1 \setminus \{c_1, c_2,\ldots, c_n, \infty\})
\end{equation}
where $c_1,c_2,\dots,c_n\in \C$ are different from each other.
It is well-known (cf.~\cite{Archinard}) that the periods of the curves \eqref{genHGC} are described in terms of \textit{Appell-Lauricella hypergeometric functions} $F_D(c_1, \ldots, c_n,\l)$.
Since $c_1,\dots, c_n$ are fixed, $F_D(c_1, \ldots, c_n,\l)$, viewed as a function of 
$\l$, satisfies the \textit{Jordan-Pochhammer differential equation} (see Section \ref{JP}).

In this paper, we consider a family of surfaces $S$ over an open subset $T_0$ of $(\mathbb{P}^1\setminus\{c_1,\dots, c_n,\infty\})^2$, where the fiber $S_t$ over $t=(\l_1,\l_2)\in T_0$ is a smooth model of the quotient surface $((C_1)_{\l_1} \times (C_2)_{\l_2})/\mu_N$, and $(C_1)_{\l_1},(C_2)_{\l_2}$ are smooth models of hypergeometric curves given by
\begin{align*}
&y_1^N = (x_1-\l_1)^A(x_1-c_1)^A\cdots(x_1-c_n)^A  \quad\text{and} \\
&y_2^N = (x_2-\l_2)^{N-A}(x_2-c_1)^{N-A}\cdots(x_2-c_n)^{N-A}.
\end{align*}
Here, the group of $N$th roots of unity $\mu_N$ acts on $C_1 \times C_2$ by 
$$\zeta \cdot ((x_1, y_1), (x_2, y_2)) = ((x_1, \zeta^{-1}y_1), (x_2, \zeta y_2)), \quad \zeta \in \mu_N. $$
After taking a suitable \'etale base change $T\to T_0$, we can construct families of $(2,1)$-cycles $\xi_{c_1}^{(i)},\xi_{c_2}^{(i)},\dots, \xi_{c_n}^{(i)}\:\:(i\in \Z/N\Z)$ on $S_t$ by a method similar to that of \cite{NS}.
Under assumptions $\gcd(N, A)=1$ and $(N+1)/(n+1) \leq A \leq (nN-1)/(n+1)$, we prove the following theorem. 

\begin{thm}[Theorem \ref{mainthm}]\label{mainthmintro}
For a very general $t\in T$, the cycles $\xi_{c_1}^{(i)},\xi_{c_2}^{(i)},\dots, \xi_{c_n}^{(i)}$ generate a subgroup of rank at least $n\cdot \varphi(N)$ in $\ch^2(S_t,1)_{\mathrm{ind}}$.
In particular, we have
$$\rk \ch^2(S_t, 1)_{\rm ind} \ge n \cdot \varphi(N). $$
Here, $\ch^2(S_t, 1)_{\rm ind}$ denotes the indecomposable part of the higher Chow group $\ch^2(S_t, 1)$. 
\end{thm}

Unlike previous works \cite{NS} and \cite{Sato}, in this paper we do not consider automorphisms of this family, since $\mathrm{Aut}(\P^1\setminus \{c_1,\dots,c_n,\infty\})$ is trivial when $n\ge 3$ and $c_1,\dots, c_n$ are general.
As a result, in the case $n=2$, which is the case essentially considered in \cite{NS}, Theorem \ref{mainthmintro} yields only the weaker bound $\operatorname{rank} \ch^2(S_t,1)_{\mathrm{ind}}\geq 2\cdot\varphi(N)$ than the one obtained in the main theorem in \cite{NS}.

The outline of the proof is the same as in \cite{NS}.
We use the {\it transcendental regulator map}, which factors through $\ch^2(S_t, 1)_{\mathrm{ind}}$, hence it suffices to evaluate the rank of the subgroup generated by the images of our cycles under this map.
We give an integral representation of a normal function $\nu_{\mathrm{tr}}(\xi_{c_j}^{(i)})$ associated to the images of $(\xi_{c_j}^{(i)})_t$ under the transcendental regulator map.
Then we explicitly compute the pairing $\langle \nu_{\mathrm{tr}}(\xi_{c_j}^{(i)}), \omega\rangle$ with a canonical relative 2-form $\omega$ by using the assumptions $(N+1)/(n+1) \leq A \leq (nN-1)/(n+1)$.
We show that the pairing is given by a local holomorphic function $F_{i, j}$ satisfying the system of differential equations
\begin{equation*}
\left\{
\begin{aligned}
\mathscr{D}_{\l_1}(F_{i, j}) &= \frac{N(1-\zeta^A)\zeta^i}{(\l_2-\l_1)^{n-1}} \left(\dfrac{A}N\right)_{n-1}
\displaystyle  \sum_{k=0}^{n-2}  \dfrac{(2-n)_k}{(Nk+A)k!}\cdot 
\frac{(c_j-\l_2)^{\frac{Nk+A}{N}}}{(c_j-\l_1)^{\frac{Nk+A}{N}}}\\
\mathscr{D}_{\l_2}(F_{i, j}) &=-\frac{N(1-\zeta^{-A})\zeta^i}{(\l_2-\l_1)^{n-1}} \left(\dfrac{N-A}N\right)_{n-1}
\displaystyle  \sum_{k=0}^{n-2}  \dfrac{(2-n)_k}{(Nk+N-A)k!}\cdot 
\frac{(c_j-\l_1)^{\frac{Nk+N-A}{N}}}{(c_j-\l_2)^{\frac{Nk+N-A}{N}}}
\end{aligned}
\right.
\end{equation*}
where $\mathscr{D}_{\l_i}$ is the \textit{Jordan-Pochhammer differential operator} (see Section \ref{JP}) defined by
$$
\mathscr{D}_{\l_i}=q_0(\l_i) \dfrac{\partial^n}{\partial\l_i^n}+p_{1,i}(\l_i) \dfrac{\partial^{n-1}}{\partial\l_i^{n-1}}+\cdots + p_{n,i}(\l_i)$$
and $(\a)_k=\a(\a+1) \cdots (\a + k-1)$ denotes the Pochhammer symbol.
Since period functions of the family $S\to T$ with respect to $\omega$  are annihilated by $\mathscr{D}_{\l_1}, \mathscr{D}_{\l_2}$, the local holomorphic functions $F_{i, j}$ generate a subgroup of rank $n\cdot \varphi(N)$ in the sheaf of holomorphic functions modulo periods.
Then Theorem \ref{mainthmintro} follows.

\subsection{Outline of the paper}
In Section \ref{SurfaceFamily}, we define the family of surfaces $S\to T$ and fix some notations.
In Section \ref{ConstChow}, we explain the construction of families of higher Chow cycles $\xi_{c_j}^{(i)}$ for $i \in \Z/N\Z$ and $j=1, \ldots, n$. 
In Section \ref{computeReg}, we explain the construction of topological $2$-chains to compute the regulator map. 
In Section \ref{PFop}, we recall the Jordan-Pochhammer differential operator, and prove that the Jordan-Pochhammer differential operator is a Picard-Fuchs differential operator for $S \to T$.
In Section \ref{mainproof}, we compute the normal functions $\nu_{\mathrm{tr}}(\xi_{c_j}^{(i)})$ and prove Theorem \ref{mainthmintro}. 

\subsection{Notations}
Throughout this paper, we fix integers $n,N$ greater than 1 and pairwise distinct complex numbers $c_1,c_2,\dots, c_n$. 
Denote the group of $N$th roots of unity by  $\mu_N$.

\section{The family of surfaces} \label{SurfaceFamily}
Let $R_0=\P^1 \setminus \{c_1,c_2,\dots, c_n,\infty\}$ and $A_0=\C\left[\l, \frac1{\l-c_1},\frac{1}{\l-c_2},\dots, \frac{1}{\l-c_n} \right]$ be the coordinate ring of $R_0$. 
Let $f_i \colon C_i \to R_0$ $(i=1, 2)$ be a family of smooth projective curves, which is a smooth projective model of 
\begin{equation*}
\begin{aligned}
y_1^N&=(x_1-\l)^{a_0}(x_1 -c_1)^{a_1} \cdots (x_1-c_n)^{a_n} && \text{and}\\
y_2^N&=(x_2-\l)^{b_0}(x_2 -c_1)^{b_1} \cdots (x_2-c_n)^{b_n}
\end{aligned}
\end{equation*}
where $a_0,a_1,\dots,a_n \in \{1, \ldots, N-1\}$ and $\gcd(N, a_0, \ldots, a_n)=1$. 
We have the natural morphism $C_i \to \P^1 \times R_0$ defined by $(x_i, y_i) \mapsto x_i$.  

Let $T_0$ be the Zariski open subset of $R_0 \times R_0$ defined by 
$$T_0=\left\{ (\l_1, \l_2) \in R_0 \times R_0 \:\: \middle| \:\:  \l_1 \neq \l_2\right\}$$ 
and $B_0$ be the coordinate ring of $T_0$. 
We have the family of surfaces $C_1\times C_2\rightarrow R_0\times R_0$.
We use the same symbol $C_1\times C_2$ for its restriction to $T_0$.

Let $\mathscr{L}$ be the line bundle on $\P^1 \times \P^1 \times T_0$ defined by
\begin{equation*}
\mathscr{L} = pr_1^*\mathscr{O}_{\P^1}(-k_1)\otimes pr_2^*\mathscr{O}_{\P^1}(-k_2)
\end{equation*}
where $pr_i\colon \P^1\times \P^1\times T_0\to\P^1$ is the $i$th projection and $k_1$ (resp.~$k_2$) is a positive integers satisfying $\sum_{i=0}^{n}a_i\le k_1N$ (resp.~$\sum_{i=0}^{n}b_i\le k_2N$).
Let $x_0,x_1$ (resp.~$y_0,y_1$) denote homogeneous coordinates of the first (resp.~second) $\P^1$-factor of $\P^1\times \P^1\times T_0$.
Let $h$ be the global section of $\mathscr{L}^{\otimes(-N)}$ defined by 
\begin{equation*}
h= x_0^{a_{\infty}}y_0^{b_{\infty}}(x_1-\l_1x_0)^{a_0}(y_1-\l_2y_0)^{b_0} \prod_{j=1}^n (x_1-c_jx_0)^{a_j} (y_1-c_jy_0)^{b_j},
\end{equation*}
where $a_\infty=k_1N-\sum_{i=0}^{n}a_i$ and $b_\infty=k_2N-\sum_{i=0}^{n}b_i$.
We define $\overline{S}$ as the cyclic covering of degree $N$ associated with $(\mathscr{L},h)$, i.e., $\overline{S}$ is the relative spectrum of an $\mathscr{O}_{\P^1\times \P^1\times T_0}$-algebra $\bigoplus_{i=0}^{N-1}\mathscr{L}^{\otimes i}$.
Here, the ring structure is defined by 
\begin{equation*}
u\cdot v = \begin{cases}
u\otimes v & (i+j<N)\\
u\otimes v\otimes h & (i+j\geq N)
\end{cases}
\end{equation*}
where $u$ and $v$ are local sections of $\mathscr{L}^{\otimes i}$ and $\mathscr{L}^{\otimes j}$, respectively.

Let $x=x_1/x_0$ and $y=y_1/y_0$ denote inhomogeneous coordinates on $\P^1\times \P^1\times T_0$ and $U=\spec B_0[x, y] \subset \P^1 \times \P^1 \times T_0$ be the corresponding affine local chart.
An affine model of $\overline{S}$ is given by
\begin{equation}\label{Sbaraffine}
z^N= (x-\l_1)^{a_0}(y-\l_2)^{b_0} \prod_{j=1}^n (x-c_j)^{a_j} (y-c_j)^{b_j} 
\end{equation}
and let $S \to \overline{S}$ be a resolution of singularities such that $S_t\to \overline{S}_t$ is the minimal resolution of $\overline{S}_t$ for each $t\in T$.
The group $\mu_N$ acts on $C_1 \times C_2$ by 
$$((x_1, y_1), (x_2, y_2)) \mapsto ((x_1, \zeta^{-1} y_1), (x_2, \zeta y_2)), \quad \zeta \in \mu_N. $$
We define a morphism $C_1\times C_2 \to \overline{S}$ by
\begin{equation}\label{CtoS}
((x_1, y_1), (x_2, y_2)) \mapsto (x,y,z) = (x_1,x_2,y_1y_2).
\end{equation}
Then this morphism descends to a birational morphism from the quotient $C_1\times C_2/\mu_N$ to $\overline{S}$.
Let $\widetilde{C_1 \times C_2}/\mu_N \to C_1 \times C_2/\mu_N$ be a resolution of singularities and $\widetilde{C_1\times C_2}$ be the fiber product of $C_1\times C_2$ and $\widetilde{C_1 \times C_2}/\mu_N$ over $C_1\times C_2/\mu_N$.
Since $\widetilde{C_1 \times C_2}/\mu_N$ is birational to $\overline{S}$, we have the birational map $\widetilde{C_1 \times C_2}/\mu_N\to S$.
Furthermore, since $S_t\to \overline{S}_t$ is minimal for each $t\in T$, this birational map is a composition of blowing-ups.
Then, we have the following commutative diagram: 
\begin{equation}\label{diagram}
  \xymatrix{
     & \ar[ld]_{\eqref{CtoS}} C_1 \times C_2 \ar[d]^{N:1\ \text{cover}} &  \ar[l]_{\text{blow-up}}  \widetilde{C_1 \times C_2} \ar[d]^{N:1\ \text{cover}}   \\
    \overline{S} \ar[d]& \ar[l] C_1 \times C_2/\mu_N & \ar[l]^{\text{blow-up}}  \widetilde{C_1 \times C_2}/\mu_N \ar[d]^{\text{blow-up}} \\
    \P^1 \times \P^1 \times T_0 & &  \ar[llu]^{\text{blow-up}}  \ar[ll] S. 
  }
\end{equation}

In the rest of the paper, we suppose that $$
a_0=a_1= \cdots =a_n=A, \quad b_0=b_1= \cdots =b_n= N-A,  
$$ where $A$ is a positive integer satisfying the following conditions.
\begin{align}
&\gcd(N, A)=1, \\ 
&\dfrac{N+1}{n+1} \leq A \leq \dfrac{nN-1}{n+1}. \label{rangeassumption}
\end{align}
We define local charts on $S$. 
The inverse image of $U\subset \P^1\times \P^1\times T_0$ by $S \to \P^1 \times \P^1 \times T_0$ is covered by two open affine subschemes: 
\begin{align*}
V&=\spec B_0[x, y, v]/(v^Nf(x) -g(y)),  \\
W&=\spec B_0[x, y, w]/(f(x) - w^Ng(y)),  
\end{align*}
where $f(x)=(x-\l_1)(x-c_1) \cdots (x-c_n)$ and $g(y)=(y-\l_2)(y-c_1) \cdots (y-c_n)$.  
These two subsets are glued by the relation $v=1/w$. 
Then the map $V \cup W$ to the affine model \eqref{Sbaraffine} of $\overline{S}$ is given by 
\begin{align*}
& V \to \overline{S}; \quad (x, y, v) \mapsto (x, y, z)=(x, y, v^{N-A} (x-\l_1)(x-c_1) \cdots (x-c_n)),  \\ 
& W \to \overline{S}; \quad (x, y, w) \mapsto (x, y, z)=(x, y, w^A (y-\l_2) (y-c_1) \cdots (y-c_n)).   
\end{align*}
By these equations, we have only one exceptional curve on $S$ above $(x, y,z)=(c_i, c_j,0)\in \overline{S}$ ($i=1, \ldots, n$) and denote it by $Q_{(c_i, c_j)}$.

\section{Construction of higher Chow cycles} \label{ConstChow}
\subsection{Preliminaries}
For a smooth variety $X$ over $\C$, let $\mathrm{CH}^2(X,1)$ be the higher Chow group of type $(2,1)$ defined by Bloch \cite{Bloch}.
It is well-known (e.g., \cite[Corollary 5.3]{MS2}) that a higher Chow cycle in $\mathrm{CH}^2(X,1)$ is represented by a formal sum
\begin{equation}\label{formalsum}
\sum_j (C_j, f_j)
\end{equation}
where $C_j$ are prime divisors on $X$ and $f_j\in \C(C_j)^\times$ are non-zero rational functions on them such that $\sum_j {\mathrm{div}}_{C_j}(f_j) = 0$ as codimension 2 cycles on $X$.

A \textit{decomposable cycle} in $\mathrm{CH}^2(X,1)$ is an element of the image of the group homomorphism
\begin{equation}\label{intersectionproduct}
\mathrm{Pic}(X)\otimes \Gamma(X,\mathscr{O}_X^\times)=\mathrm{CH}^1(X)\otimes \mathrm{CH}^1(X,1) \longrightarrow \mathrm{CH}^2(X,1)
\end{equation}
induced by the intersection product. 
The cokernel of (\ref{intersectionproduct}) is denoted by $\mathrm{CH}^2(X,1)_{\mathrm{ind}}$, 
which is called the indecomposable part of $\mathrm{CH}^2(X,1)$. 
Let $C$ be a prime divisor on $X$ and $[C]\in \mathrm{Pic}(X)$ be the class corresponding to $C$. 
The image of $[C]\otimes \alpha$ under (\ref{intersectionproduct}) is represented by $(C,\alpha|_C)$ in the presentation (\ref{formalsum}). 

\subsection{Construction of families of higher Chow cycles}
Let 
$$B= B_0\left[\sqrt[N]{c_i-\l_1}, \sqrt[N]{c_i-\l_2}\:\middle|\: i=1,2,\dots,n\right]$$ 
and $T \to T_0$ be the finite \'etale cover corresponding to $B_0 \to B$. 
We use the same notation for the base changes of the schemes in the diagram \eqref{diagram} by $T \to T_0$. 
Let $D \subset \P^1 \times \P^1 \times T$ be the irreducible closed subscheme defined by the local equation $x=y$.
We define $Z\subset S$ as the strict transformation of the inverse image of $D$ by $\overline{S} \to \P^1 \times \P^1 \times T$. 
On the local chart $V$, $Z \hookrightarrow S$ is described by the following ring homomorphism 
$$B[x, y, v]/(v^Nf(x)-g(y)) \to B [z, v]/(v^N(z-\l_1)-(z-\l_2)); \quad (x, y, v) \mapsto (z, z, v). $$
The curves $Z$ and $Q_{(c_j, c_j)}$ intersect at $N$-points: 
$$Q_{(c_j, c_j)} \cap Z= \{(x, y, v)=(c_j, c_j,  \zeta^i \sqrt[N]{c_j-\l_2}/\sqrt[N]{c_j-\l_1}) \mid i \in \Z/N\Z, j=1, \ldots, n \}. $$ 

We define rational functions $\psi^{(i)}_{c_j} \in \C(Z)$ and $\varphi^{(i)}_{c_j} \in \C(Q_{(c_j, c_j)})$ $(i \in \Z/N\Z, j=1, \ldots, n)$ by the following equations: 
\begin{align*}
& \psi^{(i)}_{c_j}=\left(v-\zeta^{i+1} \dfrac{\sqrt[N]{c_j-\l_2}}{\sqrt[N]{c_j-\l_1}} \right) \cdot \left(v- \zeta^i \dfrac{\sqrt[N]{c_j-\l_2}}{\sqrt[N]{c_j-\l_1}} \right)^{-1},  \\
&\varphi^{(i)}_{c_j}=\left(v- \zeta^{i}\dfrac{\sqrt[N]{c_j-\l_2}}{\sqrt[N]{c_j-\l_1}} \right) \cdot \left(v-\zeta^{i+1} \dfrac{\sqrt[N]{c_j-\l_2}}{\sqrt[N]{c_j-\l_1}} \right)^{-1}.  
\end{align*}
\begin{dfn} \label{cycledef}
For $i \in \Z/N\Z$ and $j=1, \ldots, n$, we define higher Chow cycles $(\xi^{(i)}_{c_j})_t\in \ch^2(S_t,1)$ by 
\begin{align*}
&(\xi_{c_j}^{(i)})_t=(Z_t, \psi_{c_j}^{(i)}) +((Q_{(c_j, c_j)})_t, \varphi_{c_j}^{(i)}),  
\end{align*}
where $Z_t$ and $(Q_{(c_j, c_j)})_t$ denote the fibers of $Z$ and $Q_{(c_j, c_j)}$ over $t\in T$.
We denote the family of higher Chow cycles $\{(\xi^{(i)}_{c_j})_t\}_{t \in T}$ by $\xi^{(i)}_{c_j}$.

\end{dfn}

\section{Computation of the regulator} \label{computeReg}
In this section, we calculate the images of our cycles
under the transcendental regulator map, which is a variant of the Beilinson regulator map.

Recall that for a smooth projective variety $X$, the Beilinson regulator is the map from $\mathrm{CH}^2(X,1)$ to the Deligne cohomology $H^3_{\mathcal D}(X,\Q(2))$.
The \textit{transcenental regulator map} is the composition of the Beilinson regulator map and the natural projection induced by $H^{2,0}(X) \hookrightarrow F^1H^2(X,\C)$.
\begin{equation}\label{transreg}
r\colon \mathrm{CH}^2(X,1)\to H^3_{\mathcal D}(X,\Q(2))\simeq
\frac{F^1H^2(X,\C)^\vee}{H_2(X,\Q)} \twoheadrightarrow \frac{H^{2,0}(X)^\vee}{H_2(X,\Q)}
\end{equation}
Here, $H^{2,0}(X)^\vee$ means the dual $\C$-linear space of $H^{2,0}(X)$ and we regard $H_2(X,\Q)$ as a $\Q$-linear subspace of $F^1H^2(X,\C)^\vee$ by the integration.

\subsection{Levine's regulator formula}
To compute the image of a higher Chow cycle represented by a formal sum \eqref{formalsum} under the transcendental regulator map $r$, we use the following formula (cf. \cite{Levine}).
For simplicity, we assume $H^1(X,\Q)=0$.

Let $D_j\to C_j$ be the normalization of the curve $C_j$ on $X$. 
Let $\mu_j\colon D_j\to X$ denote the composition of $D_j\to C_j$ and the inclusion $C_j\to X$. 
We will define a topological 1-chain $\gamma_j$ on $D_j$.
If $f_j$ is a constant, we set $\gamma_j = 0$.
If $f_j$ is not a constant, we regard $f_j$ as a finite morphism from $D_j$ to $\P^1$. 
Then we define $\gamma_j = f_j^{-1}([\infty, 0])$ where $[\infty, 0]$ is a path on $\P^1$ from $\infty$ to $0$ along the positive real axis. 
By the condition $\sum_j \mathrm{div}_{C_j}(f_j) = 0$, $\gamma = \sum_j (\mu_j)_*\gamma_j$ satisfies $\partial \gamma=0$.
In this paper, $\gamma$ is called the \textit{1-cycle associated with $\xi$}.
Since $H_1(X, \Q) = 0$, there exists a 2-chain $\Gamma$ with $\Q$-coefficients such that $\partial\Gamma = \gamma$ as 1-cycles with $\Q$-coefficients.

Then, for a non-zero holomorphic 2-form $\omega\in H^{2,0}(X)$ on $X$, we can compute the image of $\xi =\sum_j(C_j,f_j)$ under the map \eqref{transreg} as 
\begin{equation}\label{transregval}
\langle r(\xi), \omega\rangle  \equiv \int_{\Gamma}\omega \mod \mathcal{P}(\omega), 
\end{equation}
where $\mathcal{P}(\omega)$ is the $\Q$-linear subspace of $\C$ defined by
\begin{equation}\label{Pdef}
\mathcal{P}(\omega) = \left\langle \displaystyle \int_{\Gamma}\omega  \:\: \middle| \:\:\Gamma \text{ is a topological 2-cycles on }X\right\rangle_\Q, 
\end{equation}
i.e., the $\Q$-linear subspace generated by periods with respect to $\omega$.
In particular, if $\xi$ is decomposable, then $r(\xi) = 0$. 
This implies the following.

\begin{prop}\label{transregprop}
The transcendental regulator map factors through $\mathrm{CH}^2(X,1)_{\mathrm{ind}}$.
\end{prop}
\subsection{Construction of topological chains}\label{consttop}
In the rest of this section, we fix a point $t\in T$, i.e., we choose a point $(\l_1,\l_2)\in T_0$ and fix a branch of 
\begin{equation}\label{chosenbranch}
\sqrt[N]{c_j-\l_1},\quad\sqrt[N]{c_j-\l_2}\quad (j=1,2,\dots,n). 
\end{equation}
We will construct a topological 2-chain for computing the images of our cycles under the transcendental regulator map.
To use the formula \eqref{transregval}, we should prove the following.
\begin{prop}
The surface $S_t$ satisfies $H^1(S_t,\Q)=0$.
\end{prop}
\begin{proof}
We have 
\begin{equation*}
H^1(S_t,\Q)=H^1((\widetilde{C_1\times C_2/\mu_N})_t,\Q)=H^1((\widetilde{C_1\times C_2})_t,\Q)^{\mu_N} = H^1((C_1\times C_2)_t,\Q)^{\mu_N}.
\end{equation*}
Here we use the fact that $(C_1\times C_2)_t\to (\widetilde{C_1\times C_2}/\mu_N)_t$ is the quotient map and $H^1$ is stable under blowing-ups.
Thus, it is enough to show $H_1((C_1\times C_2)_t,\Q)^{\mu_N}=0$.
By the K\"unneth decomposition, it is enough to check that the invariant part of $H^1((C_i)_{\l_i},\Q)$ under the action $(x_i,y_i)\mapsto (x_i,\zeta y_i)$ is 0, but this follows from the explicit description of its action on the cohomology, which is given in \cite[Theorem 6.2]{Archinard}.
\end{proof}

Let $j\in \{2,3,\dots,n\}$ and $l\in \Z/N\Z$.
To construct a topological 2-chain computing the image of $(\xi_{c_j}^{(i+l)})_t-(\xi_{c_1}^{(i)})_t$ under the transcendental regulator map, we fix a smooth path $\gamma_j^l\colon[0,1]\to \C$ (possibly having self-intersections) from $c_1$ to $c_j$ satisfying the following conditions.
\begin{enumerate}
\item\label{i} $\gamma_j^l$ does not pass through $\l_1,\l_2,c_1,\dots,c_n$ except the starting point and the ending point.
\item There exists $\varepsilon>0$ such that $\gamma_j^l(s) = c_1+s$ and $\gamma_j^l(1-s) = c_j - s$ hold for all $s<\varepsilon$.
\item\label{l1log} The total change of $\arg(z-\l_1)/N$ along $\gamma_j^l$ coincides with $\arg(\sqrt[N]{c_j-\l_1}/\sqrt[N]{c_1-\l_1})$ where $\sqrt[N]{c_1-\l_1}$ and $\sqrt[N]{c_j-\l_1}$ are chosen complex numbers in \eqref{chosenbranch}.
In other words, there exists an analytic continuation of the function $\sqrt[N]{z-\l_1}$ along $\gamma_j^l$ such that 
\begin{equation*}
\sqrt[N]{\gamma_j^l(0)-\l_1} = \sqrt[N]{c_1-\l_1}\quad \text{ and }\quad \sqrt[N]{\gamma_j^l(1)-\l_1}= \sqrt[N]{c_j-\l_1}.
\end{equation*}
\item\label{logl22} The total change of $\arg(z-\l_2)/N$ along $\gamma_j^l$ coincides with $\arg(\zeta^l\cdot\sqrt[N]{c_j-\l_2}/\sqrt[N]{c_1-\l_2})$ where $\sqrt[N]{c_1-\l_2}$ and $\sqrt[N]{c_j-\l_2}$ are chosen complex numbers in \eqref{chosenbranch}.
In other words, there exists an analytic continuation of the function $\sqrt[N]{z-\l_2}$ along $\gamma_j^l$ such that 
\begin{equation*}
\sqrt[N]{\gamma_j^l(0)-\l_2} = \sqrt[N]{c_1-\l_2}\quad \text{ and }\quad \sqrt[N]{\gamma_j^l(1)-\l_2}= \zeta^l\cdot \sqrt[N]{c_j-\l_2}.
\end{equation*}
\end{enumerate}
We fix the branch of functions $\sqrt[N]{z-\l_1}$ and $\sqrt[N]{z-\l_2}$ along $\gamma_j^l$ as in the condition \ref{l1log} and \ref{logl22}.
We also fix branches of functions $\sqrt[N]{z-c_k}$ along $\gamma_j^l$ for $k=1,2,\dots, n$.

Let $\triangle = \{(s_1,s_2)\in \R^2\mid 0< s_2\le s_1 < 1\}$.
For $i\in \Z/N\Z$, we define the map $\mu_j^l$ as
\begin{align*}
\mu_{j}^l\colon\triangle \to V (\subset S_t);(s_1, s_2) \mapsto (x, y, v)= \left(\gamma_j^l(s_1), \gamma_j^l(s_2), 
\dfrac{\sqrt[N]{g(\gamma_j^l(s_2))}}{\sqrt[N]{f(\gamma_j^l(s_1))}}\right)
\end{align*}
where the functions 
\begin{equation}\label{branchsqrt}
\begin{aligned}
\sqrt[N]{f(\gamma_j^l(s_1))} & = \sqrt[N]{\gamma_j^l(s_1)-\l_1}\cdot\sqrt[N]{\gamma_j^l(s_1)-c_1} \cdot \cdots \cdot \sqrt[N]{\gamma_j^l(s_1)-c_n}, \\
\sqrt[N]{g(\gamma_j^l(s_2))} &= \sqrt[N]{\gamma_j^l(s_2)-\l_2}\cdot\sqrt[N]{\gamma_j^l(s_2)-c_1} \cdot \cdots \cdot \sqrt[N]{\gamma_j^l(s_2)-c_n}
\end{aligned}
\end{equation}
are defined using the chosen branches of the functions $\sqrt[N]{z-c_k}, \sqrt[N]{z-\l_1}$ and $\sqrt[N]{z-\l_2}$.

For $i\in \Z/N\Z$, let $\sigma^i\in \Aut(S_t)$ be the automorphism defined by $(x,y,v)\mapsto (x,y,\zeta^iv)$.
Then we define the compact subset $(K_{j}^l)^{(i)}$ of $S_t$ as the closure of $(\sigma^i\circ\mu_{j}^l)(\triangle)$.
Since $(K_{j}^l)^{(i)}$ has natural orientation induced by $\triangle$, we regard $(K_{j}^l)^{(i)}$ as a 2-chain on $S_t$.
By considering the limits along the diagonal line on $\triangle$, we see that 
\begin{equation*}
\left(c_1,c_1,\zeta^i\dfrac{\sqrt[N]{c_1-\l_2}}{\sqrt[N]{c_1-\l_1}}\right),\left(c_j,c_j,\zeta^{i+l}\dfrac{\sqrt[N]{c_j-\l_2}}{\sqrt[N]{c_j-\l_1}}\right)\in (K^l_j)^{(i)}.
\end{equation*}
Then we have the following.

\begin{prop}\label{regulator}
For each $\eta\in H^{2,0}(S_t)$, we have
\begin{equation*}
\langle r((\xi_{c_j}^{(i+l)})_t - (\xi_{c_1}^{(i)})_t), \eta \rangle \equiv \int_{(K^l_j)^{(i)}}\eta-\int_{(K^l_j)^{(i+1)}}\eta  \mod{\mathcal{P}(\eta)}. 
\end{equation*}
\end{prop}
\begin{proof}
\begin{figure}[h]
\centering
\includegraphics[width=12cm]{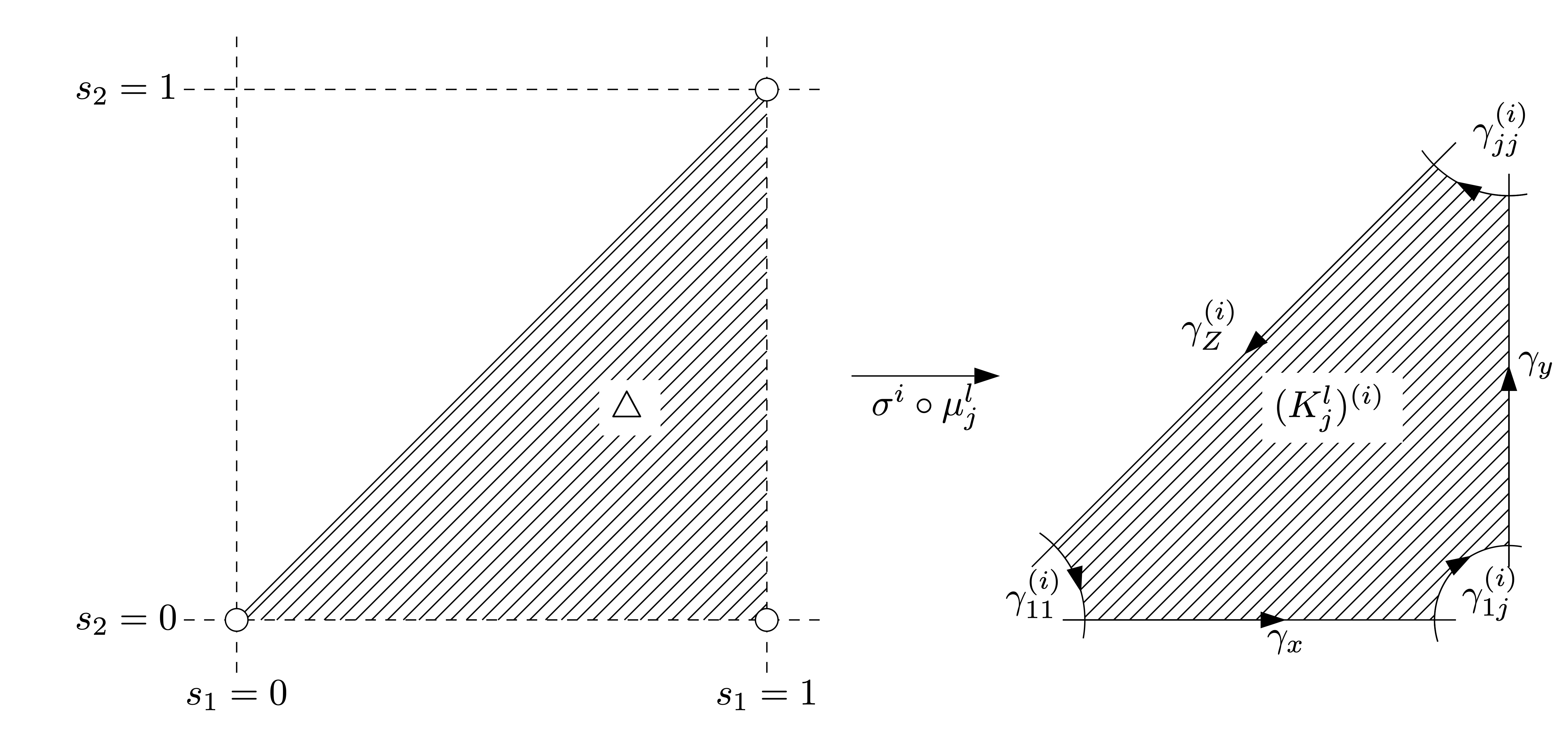}
\caption{The topological 2-chain $(K_j^l)^{(i)}$}
\label{Kfigure}
\end{figure}

The boundary of $(K_{j}^l)^{(i)}$ consists of the following paths.
See Figure \ref{Kfigure}.
\begin{enumerate}
\renewcommand{\labelenumi}{(\alph{enumi})}
\item The path $\gamma_{x}$ (resp. $\gamma_{y}$) is the intersection of $(K_j^l)^{(i)}$ and the curve defined by $y=c_1$ and $v=0$ (resp. $x=c_j$ and $w=0$).
Since $\gamma_{x}$ and $\gamma_{y}$ are contained in the branching locus of $S_t\to \P^1\times \P^1$, they does not depend on $i\in \Z/N\Z$.
\item The path $\gamma_{Z}^{(i)}$ is the closure of the image of diagonal line on $\triangle$ under $\sigma^i\circ\mu_j^l$.
It is contained in the curve $Z_t$ and a path from $(x,y,v) = \left(c_j,c_j,\zeta^{i+l}\dfrac{\sqrt[N]{c_j-\lambda_2}}{\sqrt[N]{c_j-\lambda_1}}\right)$ to $(x,y,v) = \left(c_1,c_1,\zeta^{i}\dfrac{\sqrt[N]{c_1-\lambda_2}}{\sqrt[N]{c_1-\lambda_1}}\right)$.
\item The path $\gamma_{11}^{(i)}$ (resp.$\gamma_{jj}^{(i)}, \gamma_{1j}^{(i)}$) is the intersection of $(K_j^l)^{(i)}$ and the exceptional curve $(Q_{(c_1,c_1)})_t$ (resp.~$(Q_{(c_j,c_j)})_t, (Q_{(c_j,c_1)})_t$).
\end{enumerate}
Since the starting point and the ending point of $\gamma_{1j}^{(i)}$ are the same for all $i \in \Z/N\Z$, we have $\partial(\gamma_{1j}^{(i+1)}-\gamma_{1j}^{(i)})=0$.
The exceptional curve $(Q_{(c_j,c_1)})_t$ is isomorphic to $\P^1$ and $H_1(\P^1,\Z)=0$, so there exists a 2-chain $\Gamma_{1j}^{(i)}$ on $(Q_{(c_j,c_1)})_t$ such that $\partial \Gamma_{1j}^{(i)} = \gamma_{1j}^{(i+1)}-\gamma_{1j}^{(i)}$.

Let $\widetilde{\gamma}_{Z}^{(i)}$ (resp.~$\widetilde{\gamma}^{(i)}_{11},\widetilde{\gamma}^{(i+l)}_{jj}$) be the 1-chain on $Z_t$ (resp. $(Q_{(c_1,c_1)})_t, (Q_{(c_j,c_j)})_t$) which is the pull-back of $[\infty, 0]$ on $\P^1$ by the rational function $\psi_{c_j}^{(i+l)}\cdot (\psi_{c_1}^{(i)})^{-1} \in \C(Z_t)^\times$ (resp. $\left(\varphi_{c_1}^{(i)}\right)^{-1}\in \C((Q_{(c_1,c_1)})_t)^\times, \varphi_{c_j}^{(i+l)}\in \C((Q_{(c_j,c_j)})_t)^\times$).
Then $\widetilde{\gamma}_{Z}^{(i)}+\widetilde{\gamma}^{(i)}_{11}+\widetilde{\gamma}^{(i+l)}_{jj}$ is the 1-cycle associated with $(\xi_{c_j}^{(i+l)})_t - (\xi_{c_1}^{(i)})_t$.
Considering their starting points and ending points, we have
\begin{equation*}
\partial(\gamma_{Z}^{(i+1)} + \widetilde{\gamma}_{Z}^{(i)} -\gamma_{Z}^{(i)}) = 0,\:
\partial (\gamma_{11}^{(i+1)}+ \widetilde{\gamma}_{11}^{(i)}  -\gamma_{11}^{(i)}) = 0\text{ and }
\partial (\gamma^{(i+1)}_{jj}+\widetilde{\gamma}^{(i+l)}_{jj} -  \gamma_{jj}^{(i)}) = 0.
\end{equation*}
Since $Z_t,(Q_{(0,0)})_t$, and $(Q_{(1,1)})_t$ are isomorphic to $\P^1$, we have 2-chains $\Gamma_{Z}^{(i)},\Gamma_{11}^{(i)}$, and $\Gamma_{jj}^{(i)}$ on curves $Z_t,(Q_{(c_1,c_1)})_t$, and $(Q_{(c_j,c_j)})_t$, respectively,  such that
\begin{equation*}
\partial\Gamma_{Z}^{(i)} =\gamma_{Z}^{(i+1)} + \widetilde{\gamma}_{Z}^{(i)} -\gamma_{Z}^{(i)},\:
\partial\Gamma_{11}^{(i)} =\gamma_{11}^{(i+1)}+ \widetilde{\gamma}_{11}^{(i)}  -\widetilde{\gamma}_{11}^{(i)}\text{ and }
\partial\Gamma_{jj}^{(i)} = \gamma^{(i+1)}_{jj}+\widetilde{\gamma}^{(i+l)}_{jj} -  \gamma_{jj}^{(i)}.
\end{equation*}
By above relations, we have
\begin{equation*} 
\partial ((K_{j}^l)^{(i)}-(K_j^l)^{(i+1)} + \Gamma_{Z}^{(i)} + \Gamma_{11}^{(i)}+\Gamma_{1j}^{(i)}+\Gamma_{jj}^{(i)}) =  \widetilde{\gamma}^{(i)}_{Z}+\widetilde{\gamma}^{(i)}_{11}+\widetilde{\gamma}^{(i+l)}_{jj}.
\end{equation*}
Thus, by the formula \eqref{transregval}, we have 
\begin{equation}\label{integral}
\langle r((\xi_{c_j}^{(i+l)})_t - (\xi_{c_1}^{(i)})), \eta\rangle = \int_{(K_{j}^l)^{(i)}}\eta  - \int_{(K_j^l)^{(i+1)}}\eta + \int_{\Gamma_{Z}^{(i)}}\eta + \int_{\Gamma_{11}^{(i)}}\eta + \int_{\Gamma_{1j}^{(i)}}\eta +\int_{\Gamma_{jj}^{(i)}}\eta
\end{equation}
for each $\eta\in H^{2,0}(S_t)$.
Since $\eta$ is a holomorphic 2-form, its restrictions on $(Q_{(c_1,c_1)})_t$, $(Q_{(c_j,c_j)})_t$, $(Q_{(c_j,c_1)})_t$ and $Z_t$ vanishes.
Hence, the terms after the third term in equation (\ref{integral}) are zero.
Thus, the proposition follows.
\end{proof}

\section{Picard-Fuchs differential operator}\label{PFop}
In this section, we fix a relative 2-form $\omega$ on $S\to T$ and find a \textit{Picard-Fuchs differential operator} $\mathscr{D}$, i.e., a differential operator which annihilates the period functions with respect to $\omega$.
As we stated in the introduction, the Jordan-Pochhammer differential operator becomes a Picard-Fuchs differential operator.

\subsection{A sheaf of period functions}
Let $T$ be a smooth variety and $\pi\colon X\to T$ be a smooth projective family of algebraic varieties of relative dimension $d$.
For $p\ge 1$, we fix a relative algebraic $p$-form $\eta\in \Gamma(X,\Omega_{X/T}^p)$. 
For $t\in T$, $\eta_t\in H^{0}(X_t,\Omega^{p}_{X_t})= H^{p,0}(X_t)$ denotes its restriction to the fiber $X_t$.

Hereafter, we regard $X$ and $T$ as complex manifolds with the classical topologies.
Let $\mathscr{O}_T^{\mathrm{an}}$ be the sheaf of holomorphic functions on $T$, and $\Q(d)_{X}$ be the constant sheaf on $X$ with value $\Q(d) = (2\pi i)^{-d}\Q$.
Then $R^{2d-p}\pi_*\Q(d)_X$ is a local system of finite dimensional $\Q$-vector spaces, and its local section on $U\subset T$ corresponds to a family of cohomology classes in $H^{2d-p}(X_t,\Q(d))$ for $t\in U$.
By the Poincar\'e duality $H^{2d-p}(X_t,\Q(d))\simeq H_p(X_t,\Q)$, a local section of $R^{2d-p}\pi_*\Q(d)_X$ can be described by a family of $p$-cycles $\{\Gamma_t\}_{t\in U}$ with $\Q$-coefficients.
Then we can define a morphism of sheaves
\begin{equation*}
R^{2d-p}\pi_*\Q(d)_X \to \mathscr{O}_T^{\mathrm{an}}; \quad\{\Gamma_t\}_{t\in U}\mapsto \left(t\mapsto \int_{\Gamma_t}\eta_t\right).
\end{equation*}
We denote the image sheaf of this map by $\mathcal{P}(\eta)$, which consists of period functions with respect to $\eta$.

For $t\in T$ and an open neighborhood $U$ of $t$, we have a natural evaluation map $\Gamma(U,\mathscr{O}_T^{\mathrm{an}})\rightarrow \C;f\mapsto f(t)$.
This induces the $\Q$-linear map $\Gamma(U,\mathcal{P}(\eta))\to \mathcal{P}(\eta_t)$,
where $\mathcal{P}(\eta_t)$ is a $\Q$-linear subspace defined in Section \ref{Pdef}.
Thus, we have the evaluation map
\begin{equation}\label{Qeval}
\Gamma(U,\mathscr{O}_{T}^{\mathrm{an}}/\mathcal{P}(\eta))\to \C/\mathcal{P}(\eta_t)
\end{equation}
for any open neighborhood $U$ of $t$.

\subsection{A relative 2-form $\omega$ on $S$}\label{omegasection}
Let $\omega_1$ (resp. $\omega_2$) be the relative $1$-form on $C_1$ (resp. $C_2$) over $T$ defined by 
\begin{align}\label{omega}
&\omega_1=\dfrac{dx_1}{y_1} \quad \left(\text{resp. } \omega_2=\dfrac{dx_2}{y_2}\right).
\end{align}
These are indeed regular $1$-forms on $C_1$ and $C_2$ over $T$ by the assumption \eqref{rangeassumption} (see \cite[Section 6.1]{{Archinard}}).  
We have $T$-morphisms 
$$C_1\times C_2 \xleftarrow{\text{blowing-up}} \widetilde{C_1\times C_2} \xrightarrow{\text{generically }N:1} S. $$
Let $pr_i \colon C_1 \times C_2 \to C_i$ be the $i$th projection.  
We have the relative $2$-form $pr_1^* (\omega_{1}) \wedge pr_2^*(\omega_{2})$ on $C_1\times C_2 \to T$. 
Then its pull-back to $\widetilde{C_1\times C_2}$ is stable under the covering transformation of $\widetilde{C_1\times C_2}\to S$, hence it descends to $S$.  
We denote the resulting relative $2$-form on $S$ by $\omega$, which is described as 
\begin{equation}\label{omegalocal}
\omega=\frac{dx \wedge dy}{v^{N-A} f(x)}=\frac{dx \wedge dy}{w^Ag(y)}
\end{equation}
on the local charts $V$ and $W$, respectively. 

\subsection{The differential operator $\mathscr{D}$} \label{JP}
We will find a differential operator $\mathscr{D}\colon\mathscr{O}_T^{\mathrm{an}}\to (\mathscr{O}_T^{\mathrm{an}})^{\oplus 2}$ which annihilates all sections of $\mathcal{P}(\omega)$.
The periods of our family of surface $S$ correspond to the direct product $C_1\times C_2$ of families of hypergeometric curves.
Since we fix the parameter $c_1,c_2,\dots,c_n$, the Jordan-Pochhammer differential operator is the Picard-Fuchs differential operator for hypergeometric curves $C_i\to R$.

For complex parameters $r_0,r_1,\dots,r_n\in \C$, the \textit{Jordan-Pochhammer differential operator} $L_\l$ in the variable $\l\in R_0=\P^1 \setminus \{c_1,\dots,c_n,\infty\}$ is defined by (cf.~\cite[18.4]{Ince})
$$L_\l=q_0(\l) \dfrac{d^n}{d\l^n}+p_1(\l) \dfrac{d^{n-1}}{d\l^{n-1}}+\cdots + p_{n-1}(\l) \dfrac{d}{d\l}+p_n(\l). $$
Here, the polynomials $q_0(\l)$ and $p_k(\l)\:(k=1,2,\dots,n)$ are defined by 
\begin{align}
\begin{split} \label{Jordan-Pochhammer}
&q_0(\l)=\prod_{j=1}^n(\l-c_j), \quad  q_1(\l)= q_0(\l)  \sum_{j=1}^n \dfrac{r_j}{\l-c_j}, \\
& p_k(\l)=
\begin{pmatrix}
n+r_0-2 \\
k
\end{pmatrix}
q_0^{(k)}(\l)+ 
\begin{pmatrix}
n + r_0 -2 \\
k-1
\end{pmatrix}
q_1^{(k-1)}(\l) \quad (k=1, \ldots, n).  
\end{split}
\end{align}
where $\begin{pmatrix}
\alpha \\
k
\end{pmatrix} = \dfrac{\alpha(\alpha-1)\cdots (\alpha-k+1)}{k!}$ denotes the binomial coefficients for $\alpha\in\C$ and $q^{(k)}(\lambda)= \dfrac{d^kq}{d\lambda^k}(\lambda)$.

The following lemma is useful not only for computing periods of $C_i\to R_0$ but also computing differentials of normal functions.

\begin{lem} \label{potential}
Let $H(\l,x)$ be the multivalued function defined by
$$H(\l, x)= -(r_0)_{n-1} (x-\l)^{-n+1-r_0} (x-c_1)^{1-r_1} \cdots (x-c_n)^{1-r_n}$$
where $(\alpha)_n = \alpha(\alpha+1)\cdots (\alpha+n-1)$ denotes the Pochhammer symbol. 
Then we have 
$$L_\l\left(\dfrac1{(x-\l)^{r_0}(x-c_1)^{r_1} \cdots (x-c_n)^{r_n}}\right)= \dfrac{\partial H(\l,x)}{\partial x}. $$
\end{lem}

\begin{proof}
It suffices to show that 
\begin{equation} \label{eq1}
\begin{split}
&\sum_{k=0}^n
(r_0)_{n-k}
\left\{ 
 \begin{pmatrix}
n+r_0-2 \\
k
\end{pmatrix}
q_0^{(k)}(\l)+ 
\begin{pmatrix}
n + r_0 -2 \\
k-1
\end{pmatrix}
q_1^{(k-1)}(\l)
\right\}
 (x-\l)^k \\
 &\quad =(r_0)_{n-1}\left(\dfrac{n+r_0-1}{x-\l} - \sum_{j=1}^n
 \dfrac{1-r_j}{x-c_j}\right) (x-\l) (x-c_1) \cdots (x-c_n) 
\end{split}
\end{equation}
where we regard $\begin{pmatrix}
n + r_0 -2 \\
-1
\end{pmatrix}=0$.
Since both sides of \eqref{eq1} are polynomials in $x,\lambda$ of degree at most $n$, we can uniquely expand them in the form $\dsum_{k=0}^n \dfrac{h_k(\lambda)}{k!}(x-\lambda)^k$ where each $h_k(\lambda)$ is a polynomial in $\lambda$.
Since $h_k(\lambda) = \left.\dfrac{\partial^k }{\partial x^k}\left(\dsum_{k=0}^n h_k(\lambda)(x-\lambda)^k\right)\right|_{x=\lambda}$, we will compare the $k$-th partial derivative with respect to $x$ at $x=\lambda$ of the both sides of \eqref{eq1}.
We have
\begin{align*} 
\left.\dfrac{\partial^k}{\partial x^k} (\text{LHS of } \eqref{eq1}) \right|_{x=\l}
&=(r_0)_{n-k} \left\{ 
 \begin{pmatrix}
n+r_0-2 \\
k
\end{pmatrix}
q_0^{(k)}(\l)+ 
\begin{pmatrix}
n + r_0 -2 \\
k-1
\end{pmatrix}
q_1^{(k-1)}(\l)
\right\} k!. 
\end{align*}
On the other hand, since the right-hand side of \eqref{eq1} is 
$$
(r_0)_{n-1} \left\{(n+r_0-1) q_0(x) + (x-\l)q_1(x) - (x-\l)q_0^{(1)}(x) \right\},
$$
we have
\begin{align*}
\left.\dfrac{\partial^k}{\partial x^k} (\text{RHS of } \eqref{eq1}) \right|_{x=\lambda} =(r_0)_{n-1}\left\{(n+r_0-1)q_0^{(k)}(\l) + kq_1^{(k-1)}(\l) - kq_0^{(k)}(\l) \right\}.
\end{align*}
By comparing them, we have the lemma. 
\end{proof}

Using the Jordan-Pochhammer differential operator, we will find a Picard-Fuchs differential operator $\mathscr{D}$ with respect to $\omega$ as follows.

\begin{prop} \label{PF}
Let $\omega$ be the relative $2$-form in \eqref{omega}. 
Let $\mathscr{ D}_{\l_1}, \mathscr{D}_{\l_2} \colon \mathscr{O}_T^{\rm an} \to \mathscr{O}_T^{\rm an}$ be the differential operator defined by 
\begin{align*}
&\mathscr{D}_{\l_1}=q_0(\l_1) \dfrac{\partial^n}{\partial\l_1^n}+p_{1,1}(\l_1) \dfrac{\partial^{n-1}}{\partial\l_1^{n-1}}+p_{2,1}(\l_1) \dfrac{\partial^{n-2}}{\partial\l_1^{n-2}} + \cdots + p_{n,1}(\l_1) \text{ and} \\
&\mathscr{D}_{\l_2}=q_0(\l_2) \dfrac{\partial^n}{\partial\l_2^n}+p_{1,2}(\l_2) \dfrac{\partial^{n-1}}{\partial\l_2^{n-1}}+p_{2,2}(\l_2) \dfrac{\partial^{n-2}}{\partial\l_2^{n-2}}+\cdots + p_{n,2}(\l_2),
\end{align*}
where $p_{k,1}(\l)$ (resp. $p_{k,2}(\l)$) is a polynomial $p_k(\l)$  in \eqref{Jordan-Pochhammer} for $r_0=r_1=\cdots =r_n=A$ (resp. $N-A$). 
Let $\mathscr{D}=
\begin{pmatrix} 
\mathscr{D}_{\l_1} \\
\mathscr{D}_{\l_2} 
\end{pmatrix}
\colon \mathscr{O}_T^{\rm an} \to ( \mathscr{O}_T^{\rm an})^{\oplus 2}$. 
Then for any local section $f$ of $\mathcal{P}({\omega}) \subset \mathscr{O}^{\rm an}_T$, we have $\mathscr{D}(f)=0$.
In particular,  $\mathscr{D}$ induces the morphism $\mathscr{O}_T^{\mathrm{an}}/\mathcal{P}({\omega}) \to (\mathscr{O}^{\rm an}_T)^{\oplus 2}$. 
\end{prop}

\begin{proof}
By Lemma \ref{potential}, $\mathscr{D}_{\l_1}(\omega_1)$ and $\mathscr{D}_{\l_2}(\omega_2)$ are exact forms on $C_1$ and $C_2$, i.e., these differential forms are equal to zero in the de Rham cohomology. 
Hence, for $i=1,2$, $\mathscr{D}_{\l_i}$ annihilates local sections of $\mathcal{P}(\omega_i)(\subset \mathscr{O}_T^{\mathrm{an}})$ generated by period functions with respect to $\omega_i$.

By the K\"unneth decomposition, as a subsheaf of $\mathscr{O}_T^{\mathrm{an}}$, we have 
\begin{equation*}
\mathcal{P}(\omega) = \mathcal{P}(\omega_1)\mathcal{P}(\omega_2)
\end{equation*}
where the right-hand side is the $\Q$-linear subsheaf generated by products of functions in $\mathcal{P}(\omega_1)$ and $\mathcal{P}(\omega_2)$.
Thus the proposition follows.
 
\end{proof}

\section{Proof of the main theorem} \label{mainproof}
In this section, we prove the following theorem.
\begin{thm}\label{mainthm}
For very general $t\in T$, higher Chow cycles $(\xi_{c_j}^{(i)})_t$ generate a rank $n\cdot\varphi(N)$ in $\ch^2(S_t,1)_{\mathrm{ind}}$.
In particular, we have
\begin{equation*}
\operatorname{rank}\ch^2(S_t,1)_{\mathrm{ind}} \ge n\cdot \varphi(N).
\end{equation*}
Here, $\varphi(N)$ denotes the Euler totient function.
\end{thm}

The key point of the proof is to compute normal functions associated with the images of the families of cycles $\xi_{c_j}^{(i)}$ under the transcendental regulator map, and to prove their non-vanishingness by using the differential operator $\mathscr{D}$ in Section \ref{PFop}.

\subsection{Certain normal functions}
Let $X\to T$ be a smooth projective family of algebraic surfaces.
Let $\mathcal{Q}$ be the cokernel of the following map.
\begin{equation*}
R^{2}\pi_*\Q(2)_X\to {\mathscr Hom}(\pi_*\Omega^2_{X/T},\mathscr{O}_T^{\mathrm{an}});\quad\{\Gamma_t\}_{t\in U}\mapsto \left(\eta\mapsto \left(t\mapsto \int_{\Gamma_t}\eta_t\right)\right)
\end{equation*}
Note that $\mathcal{Q}$ is a quotient of a locally free sheaf by a local system of $\Q$-vector spaces.
For a relative 2-form $\eta\in \Gamma(X,\Omega_{X/T}^2)$, a map of $\Q$-linear sheaves
\begin{equation*}
\langle-,\eta\rangle\colon\mathcal{Q}\to \mathscr{O}_T^{\mathrm{an}}/\mathcal{P}(\eta).
\end{equation*}
is induced by a natural map ${\mathscr Hom}(\pi_*\Omega^2_{X/T},\mathscr{O}_T^{\mathrm{an}})\to \mathscr{O}_T^{\mathrm{an}}$.
Furthermore, for an open subset $U$ of $t\in U$, we have the evaluation map $\mathrm{ev}_t\colon \Gamma(U,\mathcal{Q})\to H^{2,0}(X_t)^\vee/H_2(X_t,\Q)$ which fits into the following diagram.
\begin{equation*}
\begin{tikzcd}
\Gamma(U,\mathcal{Q}) \arrow[d,"\langle -{,}\eta\rangle"'] \arrow[r,"\mathrm{ev}_t"]& H^{2,0}(X_t)^\vee/H_2(X_t,\Q) \arrow[d,"\langle -{,} \eta_t\rangle "] \\
\Gamma(U,\mathscr{O}_T^{\mathrm{an}}/\mathcal{P}(\eta)) \arrow[r,"\eqref{Qeval}"] & \C/\mathcal{P}(\eta_t) 
\end{tikzcd}
\end{equation*}
The following elementary lemma is important to us.
\begin{lem}[{\cite[Lemma 2.4]{NS}}] \label{basiclem}
For a non-zero element $\nu \in \Gamma(T,\mathcal{Q})$, we have $\mathrm{ev}_t(\nu)\neq 0$ for very general $t\in T$.
\end{lem}

Suppose that we have irreducible divisors $C_j$ on $X$ which are smooth over $T$ and non-zero rational functions $f_j$ on $C_j$ whose zeros and poles are also smooth over $T$.
Assume that they satisfy the condition $\sum_j\mathrm{div}_{C_j}(f_j) = 0$.
Then we have a family of higher Chow cycles $\xi = \{\xi_t\}_{t\in S}$ such that $\xi_t\in \mathrm{CH}^2(X_t,1)$ is represented by the formal sum $\sum_{j}((C_j)_t,(f_j)_t)$.
In this situation, there exists an element $\nu_{\mathrm{tr}}(\xi)\in \Gamma(T,\mathcal{Q})$ such that 
\begin{equation}\label{existenceNF}
\mathrm{ev}_t(\nu_{\mathrm{tr}}(\xi)) = r(\xi_t)
\end{equation}
where $r\colon \ch^2(X_t,1)\to H^{2,0}(X_t)/H_2(X,\Q)$ is the transcendental regulator map (see Section \ref{computeReg} or \cite[Section 2.3]{NS} for detailed arguments).
The element $\nu_{\mathrm{tr}}(\xi)$ can be regarded as a certain type of normal function for a family of cycles $\xi$.
In this paper, we explicitly compute $\nu_{\mathrm{tr}}(\xi_{c_j}^{(i)})$ satisfying condition \eqref{existenceNF}.

\subsection{Computation of normal functions}
Let $t\in T$, $j\in \{2,3,\dots,n\}$ and $l\in \Z/N\Z$. 
We take a path $\gamma_j^l$ on $\C$ from $c_1$ to $c_j$ satisfying conditions \ref{i}-\ref{logl22} in Section \ref{computeReg}.
There exists an open neighborhood $U$ of $t$ such that $\gamma_j^l$ satisfies conditions \ref{i}-\ref{logl22} for any $t'\in U$.
Using the path $\gamma_j^l$, we can construct a 2-chain $(K_j^l)^{(i)}$ on $S_{t'}$ for each $t'\in U$.
Since $\gamma_j^l$ is common in $U$, the function
\begin{equation*}
U\ni t' \mapsto \int_{(K_j^l)^{(i)}}\eta_{t'} = \int_\triangle (\sigma^i\circ\mu_j^l)^*\eta_{t'}
\end{equation*}
defines a holomorphic function on $U$ for each relative 2-form $\eta$.
Thus, the $\mathscr{O}_T^{\mathrm{an}}$-linear map
\begin{equation}\label{locl}
\eta \mapsto \left(t'\mapsto \int_{(K_j^l)^{(i)}}\eta_{t'}-\int_{(K_j^l)^{(i+1)}}\eta_{t'}\right)
\end{equation}
defines a section $\nu_U\in \Gamma(U,\mathcal{Q})$.
By Proposition \ref{regulator}, we have
\begin{equation*}
\mathrm{ev}_{t'}(\nu_U) = r((\xi_{c_j}^{(i+l)})_{t'}-(\xi_{c_1}^{(i)})_{t'}) =\mathrm{ev}_{t'}\left(\nu_{\mathrm{tr}}(\xi_{c_j}^{(i+l)}-\xi_{c_1}^{(i)})\right)
\end{equation*}
for all $t'\in U$.
Then by Lemma \ref{basiclem}, we have $\nu_U = \nu_{\mathrm{tr}}(\xi_{c_j}^{(i+l)}-\xi_{c_1}^{(i)})|_U$.
Thus, \eqref{locl} is the local description of $\nu_{\mathrm{tr}}(\xi_{c_j}^{(i)}-\xi_{c_1}^{(i)})$.

We can compute $\nu_{\mathrm{tr}}(\xi_{c_j}^{(i)})$ by \eqref{locl}.
To prove their non-triviality, we consider the pairing with the relative 2-form $\omega$ in Section \ref{omegasection}.
By Proposition \ref{PF}, the differential operator $\mathscr{D}\colon\mathscr{O}_T^{\mathrm{an}}\to (\mathscr{O}_T^{\mathrm{an}})^{\oplus 2}$ annihilates all local sections of $\mathcal{P}(\omega)$.
Thus, $\mathscr{D}$ induces the $\Q$-linear map
\begin{equation}\label{DQmap}
\mathscr{D}\colon \mathscr{O}_T^{\mathrm{an}}/\mathcal{P}(\omega) \to (\mathscr{O}_T^{\mathrm{an}})^{\oplus 2}
\end{equation}
We will compute the explicit images of $\langle \nu_{\mathrm{tr}}(\xi_{c_j}^{(i+l)}-\xi_{c_1}^{(i)}),\omega \rangle$ under the map \eqref{DQmap}.

\begin{lem}\label{covering}
Let $t\in T$ and $(K_j^l)^{(i)}$ be topological 2-chains in Section \ref{consttop}.
For $i \in \Z/N\Z$, we have 
\begin{equation*}
\int_{(K_j^l)^{(i)}}\omega_t = \zeta^{Ai}\int_{(K_j^l)^{(0)}} \omega_t.
\end{equation*}
\end{lem}
\begin{proof}
The topological 2-chain $(K_j^l)^{(i)}$ satisfies $\sigma^i((K_{j}^l)^{(0)}) = (K_j^l)^{(i)}$, where $\sigma^i\in\Aut(S_t)$ is the automorphism defined by $(x,y,v)\mapsto (x,y,\zeta^iv)$.
By the local description of $\omega$ in \eqref{omegalocal}, we have $(\sigma^i)^*\omega =\zeta^{Ai}\omega$.
Thus, the equality follows as
\begin{equation*}
\int_{(K_j^l)^{(i)}}\omega_t = \int_{\sigma^i((K_{j}^l)^{(0)})}\omega_t = \int_{(K_{j}^l)^{(0)}}(\sigma^i)^*\omega_t = \zeta^{Ai}\int_{(K_j^l)^{(0)}} \omega_t.
\end{equation*}
\end{proof}

Using Lemma \ref{covering}, $\langle \nu_{\mathrm{tr}}(\xi_{c_j}^{(i+l)}-\xi_{c_1}^{(i)}),\omega \rangle \in \Gamma(T,\mathscr{O}_T^{\mathrm{an}}/\mathcal{P}(\omega))$ is represented by the local holomorphic function 
\begin{equation}\label{afterpair}
t\mapsto \zeta^{Ai}(1-\zeta^A)\int_{(K_j^l)^{(0)}}\omega_t.
\end{equation}
Using this integral expression, we compute $\mathscr{D}(\langle \nu_{\mathrm{tr}}(\xi_{c_j}^{(i)}),\omega \rangle)$ as follows.

\begin{thm} \label{main:1}
For $i \in \Z/N\Z$, the images of $\nu_{\rm tr}(\xi_{c_j}^{(i)})$ $(j=1, \ldots, n)$ under the Picard-Fuchs differential operator $\mathscr{D}$ are as follows: 
\begin{align*}
&\mathscr{D}(\nu_{\rm tr}(\langle \xi_{c_j}^{(i)}), \omega \rangle)= 
\frac{N(1-\zeta^A)\zeta^{Ai}}{(\l_2-\l_1)^{n-1}}
\sum_{k=0}^{n-2}
\dfrac{(2-n)_k}{k!}
\begin{pmatrix}
  \dfrac{\left(\frac{A}{N} \right)_{n-1}}{(Nk+A)} \cdot
 \frac{(c_j-\l_2)^{\frac{Nk+A}{N}}}{(c_j-\l_1)^{\frac{Nk+A}{N}}}\\
- 
  \dfrac{\left(\frac{N-A}N\right)_{n-1}}{(Nk+N-A)} \cdot \frac{(c_j-\l_1)^{\frac{Nk+N-A}{N}}}{(c_j-\l_2)^{\frac{Nk+N-A}{N}}}
\end{pmatrix}. 
\end{align*}
\end{thm}

\begin{proof}
It is enough to calculate the image of the local holomorphic function in \eqref{afterpair} under the differential operator $\mathscr{D}$.
Note that the image is a pair of global holomorphic functions, though \eqref{afterpair} itself is a multivalued function.
We only compute the first component, because the computation on the second component is similar.

First, we show 
\begin{align} \label{eq11}
\begin{split}
&\mathscr{D}_{\l_1}(\langle \nu_{\rm tr}(\xi_{c_j}^{(i+l)} - \xi_{c_1}^{(i)}), \omega \rangle) \\
&=
\frac{N(1-\zeta^A)\zeta^{Ai}}{(\l_2-\l_1)^{n-1}} \left(\dfrac{A}N\right)_{n-1}
\displaystyle  \sum_{k=0}^{n-2}  \dfrac{(2-n)_k}{(Nk+A)k!} \cdot  
\left(
\zeta^{Al} \frac{(c_j-\l_2)^{\frac{Nk+A}{N}}}{(c_j-\l_1)^{\frac{Nk+A}{N}}}- \frac{(c_1-\l_2)^{\frac{Nk+A}{N}}}{(c_1-\l_1)^{\frac{Nk+A}{N}}}
\right). 
\end{split}
\end{align}
By the local description of $\omega_t$, we have
\begin{equation*}
\mathscr{D}_{\l_1} \left(\int_{(K_j^l)^{(0)}} \omega_t\right) = 
\mathscr{D}_{\l_1}\left(\int_{\triangle}(\mu_j^l)^*\left(\dfrac{dx\wedge dy}{v^{N-A}f(x)}\right)\right)
=\mathscr{D}_{\l_1} \left(\int_{\triangle} \dfrac{(\gamma_j^l(s_1))'(\gamma_j^l(s_2))'ds_1\wedge ds_2}{\sqrt[N]{f(\gamma_j^l(s_1))}^A \sqrt[N]{g(\gamma_j^l(s_2))}^{N-A}}\right)
\end{equation*}
where branches of functions $\sqrt[N]{f(\gamma_j^l(s_1))}$ and $\sqrt[N]{g(\gamma_j^l(s_2))}$ are determined as in \eqref{branchsqrt}.

For $\varepsilon>0$, let $\triangle_\varepsilon=\{(s_1,s_2)\in \R^2\mid \varepsilon \le s_2\le s_1\leq 1-\varepsilon\}$.
Then, by standard analytic arguments using Lebesgue's dominated convergence theorem, we have
\begin{align*}
=\mathscr{D}_{\l_1} \left(\lim_{\varepsilon\rightarrow 0}\int_{\triangle_{\varepsilon}} \dfrac{(\gamma_j^l(s_1))'(\gamma^l_j(s_2))'ds_1\wedge ds_2}{\sqrt[N]{f(\gamma_j^l(s_1))}^A \sqrt[N]{g(\gamma_j^l(s_2))}^{N-A}}\right)= \lim_{\varepsilon\rightarrow 0} \int_{\triangle_\varepsilon } \mathscr{D}_{\l_1}\left(\dfrac{(\gamma_j^l(s_1))'(\gamma_j^l(s_2))'ds_1\wedge ds_2}{\sqrt[N]{f(\gamma_j^l(s_1))}^A \sqrt[N]{g(\gamma_j^l(s_2))}^{N-A}}\right).
\end{align*}
Let $H(\l_1,x) = -\left(\frac{A}{N}\right)_{n-1}(x-\l_1)^{-n+1-\frac{A}{N}}(x-c_1)^{1-\frac{A}{N}}\cdots (x-c_n)^{1-\frac{A}{N}}$ be the multivalued function defined in Lemma \ref{potential}.
Since the branches of functions $\sqrt[N]{x-\l_1}$ and $\sqrt[N]{x-c_i}$ is already chosen along $\gamma^l_j$, the branch of $H(\l_1,x)$ along $\gamma^l_j$ is already fixed.
Since we have $\mathscr{D}_{\l_1}\left(1/\sqrt[N]{f(x)}^{A}\right) = \frac{\partial}{\partial x}H(x,\l_1)$ by Lemma \ref{potential}, we have
\begin{equation*}
=\lim_{\varepsilon\rightarrow 0} \int_{\triangle_\varepsilon } 
\dfrac{\partial}{\partial s_1}\left(\dfrac{H(\l_1,\gamma_j^l(s_1))\cdot(\gamma_j^l(s_2))'}{\sqrt[N]{g(\gamma_j^l(s_2))}^{N-A}}\right)ds_1\wedge ds_2 = 
\lim_{\varepsilon\rightarrow 0} \int_{\triangle_\varepsilon }d\left(\dfrac{H(\l_1,\gamma_j^l(s_1))\cdot(\gamma^l_j(s_2))'ds_2}{\sqrt[N]{g(\gamma_j^l(s_2))}^{N-A}}\right)
\end{equation*}
Then we use the Stokes theorem.
Note that integration along $\partial\triangle_\varepsilon$ converges to $0$ except the diagonal line $s_1=s_2$, so we have
\begin{align*}
&=\lim_{\varepsilon\rightarrow 0} \int_{\partial\triangle_\epsilon }\dfrac{H(\l_1,\gamma_j^l(s_1))\cdot(\gamma^l_j(s_2))'ds_2}{\sqrt[N]{g(\gamma_j^l(s_2))}^{N-A}} = \int_{s = 1}^{s=0}\dfrac{H(\l_1,\gamma_j^l(s))\cdot(\gamma^l_j(s))'ds}{\sqrt[N]{g(\gamma_j^l(s))}^{N-A}} \\
&= \left(\frac{A}{N}\right)_{n-1}\int_{s=0}^{s=1}\dfrac{(\gamma^l_j(s))'ds}{(\gamma_j^l(s)-\l_1)^{n-1}\cdot \sqrt[N]{\gamma_j^l(s)-\l_1}^A\cdot \sqrt[N]{\gamma_j^l(s)-\l_2}^{N-A}}
\end{align*}
Let $u$ be the new path on $\C$ defined by $u(s)=\frac{\sqrt[N]{\gamma_j^l(s)-\l_2}}{\sqrt[N]{\gamma_j^l(s)-\l_1}}$.
By the condition \ref{l1log} and condition \ref{logl22} on $\gamma_j^l$, $u$ is the path from $\frac{\sqrt[N]{c_1-\l_2}}{\sqrt[N]{c_1-\l_1}}$ to $\zeta^l\frac{\sqrt[N]{c_j-\l_2}}{\sqrt[N]{c_j-\l_1}}$.
Then we have
\begin{align*}
&=\frac{-N}{(\l_1-\l_2)^{n-1}} \left(\dfrac{A}N\right)_{n-1} \int_{s=0}^{s=1} (u(s)^N-1)^{n-2}u(s)^{A-1}u'(s)ds \\
&=\frac{-N}{(\l_1-\l_2)^{n-1}} \left(\dfrac{A}N\right)_{n-1}\int_{\frac{\sqrt[N]{c_1-\l_2}}{\sqrt[N]{c_1-\l_1}}}^{\zeta^l\frac{\sqrt[N]{c_j-\l_2}}{\sqrt[N]{c_j-\l_1}}} (u^N-1)^{n-2}u^{A-1}du \\
&= \frac{N}{(\l_2-\l_1)^{n-1}} \left(\dfrac{A}N\right)_{n-1}
 \sum_{k=0}^{n-2}  \dfrac{(2-n)_k}{(Nk+A)k!} 
  \left(\zeta^{Al} \frac{(c_j-\l_2)^{\frac{Nk+A}{N}}}{(c_j-\l_1)^{\frac{Nk+A}{N}}}- \frac{(c_1-\l_2)^{\frac{Nk+A}{N}}}{(c_1-\l_1)^{\frac{Nk+A}{N}}}\right). 
\end{align*}
We use the binomial expansion and the relation $\binom{a}{k}=\frac{(-1)^k(-a)_k}{k!}$ in the last equality.
Thus, we have \eqref{eq11}. 

Secondly, we will prove the theorem. 
By the definition of $\xi_{c_j}^{(i)}$, we have $\sum_{l\in \Z/N\Z}\xi_{c_j}^{(i+l)}=0$.
Thus, we have
\begin{equation*}
\sum_{l\in\Z/N\Z}\mathscr{D}(\langle \nu_{\rm tr}(\xi_{c_j}^{(i+l)} - \xi_{c_1}^{(i)}), \omega \rangle)=-N\mathscr{D}(\langle
\nu_{\rm tr}(\xi_{c_1}^{(i)}), \omega\rangle).
\end{equation*}
Since $\gcd(N, A)=1$, the statement follows from \eqref{eq11}.   
\end{proof}

\subsection{Proof of Theorem \ref{mainthm}}

Let $\Xi$ be the subgroup of $\Gamma(T,\mathcal{Q})$ generated by
\begin{equation*}
\nu_{\mathrm{tr}}(\xi_{c_j}^{(i)})\quad (i\in \Z/N\Z,\ j=1,2,\dots,n).
\end{equation*}
First, we compute the rank of $\Xi$.
By Theorem \ref{main:1}, the rational functions $\mathscr{D}(\nu_{\rm tr}(\langle \xi_{c_j}^{(0)}), \omega \rangle)$ $(j=1, \ldots, n)$ are linearly independent over $\C$, in particular, linearly independent over $\Q(\zeta)$.
Furthermore, since $\mathscr{D}(\nu_{\rm tr}(\langle \xi_{c_j}^{(0)}), \omega \rangle)$ is non-torsion for each $j=1,2,\dots, n$, $\mathscr{D}(\nu_{\rm tr}(\langle \xi_{c_j}^{(i)}), \omega \rangle)\quad (i\in \Z/N\Z)$ generates a free $\Z[\zeta]$-module of rank 1, whose rank as an abelian group is equal to $\varphi(N)$. 
Thus, the rank of $\Xi$ is equal or greater than $n\cdot \varphi(N)$.

By Lemma \ref{basiclem}, we have $\operatorname{rank} \mathrm{ev}_t(\Xi) = \operatorname{rank} \Xi$ for very general $t\in T$.
Furthermore, by \eqref{existenceNF}, we have
\begin{equation*}
\mathrm{ev}_t(\Xi) = \langle r((\xi_{c_j}^{(i)})_t)\mid i\in \Z/N\Z,\ j=1,2,\dots, n\rangle.
\end{equation*}
Therefore, by Proposition \ref{transregprop}, the cycles $(\xi_{c_j}^{(i)})_t$ generate a subgroup of rank at least $n \cdot \varphi(N)$ in $\ch^2(S_t,1)_{\mathrm{ind}}$.

\section*{Acknowledgment}
The first author is supported by Waseda University Grant for Early Career Researchers (Project number: 2025E-041). 
The second author is supported by JSPS KAKENHI 21H00971.

\end{document}